\documentclass{amsart}
\usepackage{graphicx} 
\usepackage{yufei}

\title{Extremal values for the square energies of graphs}
\author{Shengtong Zhang}
\date{September 2024}

\address{Department of Mathematics, Stanford University, Stanford, CA 94305, USA}
\email{stzh1555@stanford.edu}

\begin{document}

\begin{abstract}
Let $G$ be a graph with $n$ non-isolated vertices and $m$ edges. The positive / negative square energies of $G$, denoted $s^+(G)$ / $s^-(G)$, are defined as the sum of squares of the positive / negative eigenvalues of the adjacency matrix $A_G$ of $G$. In this work, we provide several new tools for studying square energy, encompassing semi-definite optimization, graph operations, and surplus. Using our tools, we prove the following results on the extremal values of $s^{\pm}(G)$ with a given number of vertices and edges.

1. We have $\min(s^+(G), s^-(G)) \geq n - \gamma \geq \frac{n}{2}$, where $\gamma$ is the domination number of $G$. This verifies a conjecture of Elphick, Farber, Goldberg and Wocjan up to a constant, and proves a weaker version of this conjecture introduced by Elphick and Linz.

2. We have $s^+(G) \geq m^{6/7 - o(1)}$ and $s^-(G) = \Omega(m^{1/2})$, with both exponents being optimal.
\end{abstract}

\maketitle

\section{Introduction}
For a graph $G$, let $\lambda_1 \geq \lambda_2 \geq \cdots \geq \lambda_n$ be the eigenvalues of its adjacency matrix. They are also known as the \emph{spectrum} of $G$. The \emph{energy} of $G$ is defined as the sum over the absolute values of all the eigenvalues
$$\cE(G) = \sum_{i = 1}^n \abs{\lambda_i}.$$
Motivated by applications in molecular chemistry, the energy has attracted considerable attention. We refer the reader to \cite{energy} for some classical bounds on the energy of a graph.

Here we are primarily interested in an $L^2$-analog of the energy. The \emph{positive square energy} of $G$ is defined as the sum of squares of the positive eigenvalues
$$s^+(G) = \sum_{i: \lambda_i \geq 0} \lambda_i^2.$$
Similarly, the \emph{negative square energy} of $G$ is defined as the sum of squares of the negative eigenvalues
$$s^-(G) = \sum_{i: \lambda_i \leq 0} \lambda_i^2.$$
These notions of square energy were introduced by Elphick and Wocjan in \cite{EW13} to provide a bound on the chromatic number. Ando and Lin \cite{AL15} gave a bound on the chromatic number of a graph in terms of the square energy
$$\max\left(\frac{s^+(G)}{s^-(G)}, \frac{s^-(G)}{s^+(G)}\right) + 1 \leq \chi(G).$$
Recently, Guo and Spiro \cite{GS24} showed that $\chi(G)$ can be replaced with the fractional chromatic number $\chi_f(G)$, and Coutinho and Spier \cite{Coutinho2023} further replaced $\chi_f(G)$ with the vector chromatic number $\chi_v(G)$. The square energy is also related to the Bollob\'{a}s-Nikiforov conjecture \cite{BN07}, a hypothetical strengthening of Tur\'{a}n's theorem in spectral graph theory.

In this paper, we study the extremal values of $s^+(G)$ and $s^-(G)$. First, we consider how small $s^+(G)$ and $s^-(G)$ can be given the number of vertices. The following intriguing conjecture of Elphick, Farber, Goldberg and Wocjan \cite{EFGW16} predicts a tight lower bound for $s^{\pm}(G)$. 
\begin{conjecture} 
\label{conjecture:EFGW}
    Let $G$ be a connected graph on $n$ vertices. Then $\min(s^+(G), s^-(G)) \geq n - 1$.
\end{conjecture}
This conjecture is tight at the two extremes of connected graph: we have $s^+(G) = s^-(G) = n - 1$ when $G$ is a tree, and $s^-(G) = n - 1$ when $G = K_n$. 

This conjecture has attracted some attention from the spectral graph theory community; for example, it is listed as the first conjecture in a recent survey \cite{liu23unsolved} by Liu and Ning. In \cite{Abiad23, EFGW16, EL24}, many useful partial results and computational data are gathered for this conjecture. For example, the conjecture is shown for bipartite graphs, in which case $s^+(G) = s^-(G) = \abs{E(G)}$ due to symmetry of spectrum. It is proved for regular graphs using Ando-Lin's chromatic number bound together with Brooks' theorem. It is also verified for special classes of graphs, such as the barbell graph and all connected graphs with up to $10$ vertices. However, before our work, the best general lower bound on $\min(s^-(G), s^+(G))$ is $\sqrt{n}$ in \cite{EL24}.

Next, we consider what happens if we fix the number of edges instead. This is a natural problem, since the sum of $s^+(G)$ and $s^-(G)$ is exactly twice the number of edges.
\begin{problem}
\label{problem:edges}
    Let $G$ be a graph with $m$ edges. What are the minimum values of $s^+(G)$ and $s^-(G)$?
\end{problem}
Intuitively, $s^-(G)$ should be minimized when $G$ is a clique, in which case we have $s^-(G) = \Theta(\sqrt{m})$. However, the minimum value of $s^+(G)$ seems less straightfoward: at both extremes of connected graphs, i.e. trees and cliques, $s^+(G)$ is equal to $\Theta(m)$.
\subsection{Results on \cref{conjecture:EFGW}}
Towards \cref{conjecture:EFGW}, we prove a lower bound on $s^{\pm}(G)$ in terms of the number of vertices and the domination number. Recall that a \emph{dominating set} $D$ in $G$ is a set of vertices such that every vertex of $G$ either lies in $D$ or is the neighbor of a vertex in $D$. The \emph{domination number} $\gamma(G)$ of $G$ is the size of the smallest dominating set in $G$.
\begin{theorem}
\label{thm:main-result-1}    
Let $G$ be a graph on $n$ vertices. Then we have
$$\min(s^+(G), s^-(G)) \geq n - \gamma(G).$$
\end{theorem}
Using various upper bounds on $\gamma(G)$, we record several corollaries of interest. For example, if $G$ has no isolated vertices, then $\gamma(G) \leq \frac{n}{2}$. Thus, we confirm \cref{conjecture:EFGW} up to a constant.
\begin{corollary}
\label{cor:main-1-1}
    Let $G$ be any graph on $n$ vertices with no isolated vertices. Then $\min(s^+(G), s^-(G)) \geq n / 2$.    
\end{corollary}
This corollary as stated is tight when $n$ is even. If $G$ is a disjoint union of $n / 2$ edges, then we have $s^+(G) = s^-(G) = n / 2$.

In addition, we can relate $\gamma(G)$ to the inertia of $A_G$, resolving a weaker version of \cref{conjecture:EFGW} proposed by Elphick-Linz in \cite{EL24}.
\begin{corollary}[\cite{EL24}, Conjecture 11]
\label{thm:main-result-1-weak} 
Let $G$ be a graph with no isolated vertex, and let $(n^+, n^0, n^-)$ be the inertia of the adjacency matrix of $G$. Then we have 
$$\min(s^+(G), s^-(G)) \geq \max(n^+, n^0, n^-).$$
\end{corollary}
This corollary  implies that the ``mean square" of the positive/negative eigenvalues of any graph is at least $1$. The analogous statement for energy, i.e. $\cE(G) \geq 2\min(n^+, n^-)$, is an open conjecture proposed by the automated program ``Written on the Wall" \cite{Abiad23}, \cite[Section 17]{liu23unsolved}.

To prove \cref{thm:main-result-1}, we introduce two new tools. Our first tool shows that the square energies are ``super-additive".
\begin{theorem}
\label{thm:main-1}
Let $G$ be any graph with vertex set $V$, and let $V = U \sqcup W$ be any partition of the vertex set. Then we have
$$s^+(G) \geq s^+(G[U]) + s^+(G[W]) \text{ and } s^-(G) \geq s^-(G[U]) + s^-(G[W]).$$
\end{theorem}
For comparison, \cite{Abiad23} only stated and used the weaker bound $s^{\pm}(G) \geq s^{\pm}(G[U])$.  We remark that the analogous statement for energy
$$\cE(G) \geq \cE(G[U]) + \cE(G[W])$$
is also true (see e.g. \cite{Thomspon75}, \cite[Eq. 6]{AGO09}).

This theorem can be rephrased as follows.
\begin{corollary}
\label{cor:main-1-restate}
Let $G$ be any graph with vertex set $V$, and let $V = V_1\sqcup \cdots \sqcup V_k$ be any partition of the vertex set. Then we have
$$s^+(G) \geq \sum_{i = 1}^k s^+(G[V_i])\text{ and } s^-(G) \geq  \sum_{i = 1}^k s^-(G[V_i]).$$    
In particular, we also have
$$\min(s^+(G), s^-(G)) \geq \sum_{i = 1}^k \min(s^+(G[V_i]), s^-(G[V_i])).$$
\end{corollary}
This tool itself already implies \cref{cor:main-1-1}. We now record a few other applications.

A graph is \emph{unicyclic} if it is connected and contains exactly one cycle. In other words, it is obtained by adding an edge to a tree. In \cite{Abiad23}, Abiad et. al. studied and verified \cref{thm:main-1} for many unicyclic graphs, such as the special graphs $U_{n, 3}$ and unicyclic graphs with a long cycle of length $\Omega(\sqrt{n})$. They argue that unicyclic graphs are particularly challenging because they are close to the graphs (i.e. trees) that achieve equality in the bound. We prove a new lower bound on $\min(s^+(G), s^-(G))$ for unicyclic $G$, improving upon the previous result \cite[Corollary 3.5]{Abiad23} when the cycle in $G$ is short.
\begin{corollary}
\label{cor:main-1-2}
    For any unicyclic graph $G$ on $n$ vertices, we have $\min(s^+(G), s^-(G)) \geq n - 2.$
\end{corollary}
As another application of \cref{thm:main-1}, we can show that $G$ satisfies \cref{conjecture:EFGW} whenever $G$ admits a nice enough decomposition. For example, we prove that
\begin{corollary}
    \label{cor:main-1-3}
    For any graph $G$ on $n$ vertices, suppose the vertex set of $G$ can be decomposed into subsets $V_1, \cdots, V_k$ such that for each $i$, the induced subgraph $G[V_i]$ is a bipartite graph with at least $\abs{V_i}$ edges (e.g. even cycles). Then $\min(s^+(G), s^-(G)) \geq n$. 
\end{corollary}
Our second tool arises out of an attempt to prove \cref{conjecture:EFGW} via induction. If there exists a vertex $v$ of $G$ such that $G \backslash \{v\}$ is connected and $s^{\pm}(G) \geq s^{\pm}(G \backslash \{v\}) + 1$, then we can reduce \cref{conjecture:EFGW} for $G$ to \cref{conjecture:EFGW} for $G \backslash \{v\}$. The next theorem sheds light to one situation in which we may do so.
\begin{theorem}
\label{thm:main-2}
Let $G$ be any graph. Suppose $U$ is a set of three vertices in $G$ such that $G[U]$ is isomorphic to the three-vertex path $P_3$. Then there exists a vertex $u \in U$ such that $s^-(G) > s^-(G \backslash \{u\}) + 1$. The same holds if we replace $s^-$ with $s^+$.
\end{theorem}
As applications of this tool, we establish \cref{conjecture:EFGW} for two new classes of graphs that seem out of reach with previous methods. First, suppose $G$ has a \emph{dominating vertex}, which is a vertex adjacent to every other vertex. In \cite{Abiad23} Abiad et. al. observe that since $K_{1, n - 1}$ is a subgraph of $G$, we have $\lambda_1 \geq \sqrt{n - 1}$, so $s^+(G) \geq \lambda_1^2 = n - 1$. However, no nontrivial bound is available for the least eigenvalue of $G$, since for $G = K_{n}$ the least eigenvalue of $G$ is $-1$. In \cite{Abiad23}, there is no discussion of $s^-(G)$ for general graphs $G$ with a dominating vertex. 

Using \cref{thm:main-2}, we verify \cref{conjecture:EFGW} for all graphs with a dominating vertex.
\begin{corollary}
\label{cor:main-2-1}    
Let $G$ be a graph on $n$ vertices with a dominating vertex. Then $s^-(G) \geq n - 1$, with equality if and only if $G \cong K_{1, n - 1}$ or $G \cong K_{n}$.
\end{corollary}
Combining this result with our first tool, we complete the proof of \cref{thm:main-result-1}.

As another application, we resolve \cref{conjecture:EFGW} for a large class of unicyclic graphs with a short cycle, including the graphs $H_n^k$ in \cite[Section 3.3]{Abiad23} with $k \geq 5$.
\begin{corollary}
\label{cor:main-2-3}
    Let $G$ be a non-bipartite unicyclic graph with $n$ vertices, such that the cycle in $G$ has length at least $5$ and contains three consecutive vertices with degree $2$. Then we have $\min(s^+(G), s^-(G)) > n - 1$.
    
    In particular, let $T$ be any tree, and let $C_k$ be an odd cycle with $k \geq 5$. Let $G$ be a unicyclic graph obtained by identifying any vertex of $T$ with a vertex of $C_k$. Then $\min(s^+(G), s^-(G)) > n - 1$.
\end{corollary}
\subsection{Results on \cref{problem:edges}}
We reveal a curious relation between $s^+(G)$ and the number of edges $m$.
\begin{theorem}
    \label{thm:main-result-2}
    Let $G$ be a graph with $m$ edges. Then we have
    $$s^+(G) = \Omega\left(\frac{m^{6/7}}{(\log\log m)^{8/3}}\right) \text{ and }s^-(G) = \Omega(m^{1/2}).$$
    In addition, both exponents are tight. There exist graphs $G$ with $s^+(G) = O(m^{6/7})$, and graphs $G$ with $s^-(G) = O(m^{1/2})$.
\end{theorem}
This result shows ``asymmetry" between $s^+(G)$ and $s^-(G)$, a theme introduced by Elphick-Linz in \cite{EL24}. Let us mention a matrix theoretic formulation of this theorem, which follows from a connection with semi-definite programming,  \cref{lem:energy-sdp}.
\begin{corollary}
    Let $A$ be a symmetric $\{0, 1\}$-matrix with zero diagonal, such that $m$ off-diagonal entries are non-zero. Let $M$ be any positive semi-definite matrix with the same dimension as $A$. Then we have
    $$\norm{A + M}_F = \Omega\left(\frac{m^{3/7}}{(\log\log m)^{4/3}}\right) \text{ and }\norm{A - M}_F = \Omega(m^{1/4}).$$
    where $\norm{\cdot}_F$ is the Frobenius norm.    
\end{corollary}
Since our methods work well with weighted graphs, we can relax our assumptions on $A$. For example, we can assume that the nonzero entries of $A$ are in $[0.5, 2]$, and even remove the symmetry condition on $A$ using $\norm{A + M}_F \geq \norm{\frac{1}{2}(A + A^T) + M}_F$.

We first show how to construct graphs $G$ with $s^+(G) = O(m^{6/7})$. This construction was found by Elphick and Linz using generalized quadrangles. A construction with slightly different parameters is featured in \cite{RST23}.
\begin{theorem}[\cite{EL24}, Theorem 4]
    \label{thm:quadrangle-construction}
    For any prime power $q$, there exists a graph $G$ on $(q + 1)(q^3 + 1)$ vertices with spectrum $k^1r^fa^g$, where
    $$k = q(q^2 + 1), r = q - 1, a = -q^2-1, f = q^2(q^2 + 1), g = q(q^2 - q + 1).$$
    In particular, this graph satisfies $s^+(G) \sim 2q^6$, $s^-(G) \sim q^7$, and $m \sim q^7$.
\end{theorem}
A recurring theme in this part of the discussion is that this graph has the ``minimum" $s^+(G)$ in some senses.

We introduce the ``triangle-counting" method to study \cref{problem:edges}. The third moment of the adjacency eigenvalues 
$$\sum_{i = 1}^n \lambda_i^3$$
is six times the number of triangles in the graph $G$. This fact can be exploited to prove non-trivial propositions about the eigenvalues. For example, in \cite{JTYZZ}, Jiang et. al. used this method to prove a bound on the eigenvalue multiplicity of signed graphs. In more recent works \cite{KP24, LNW21}, Lin-Ning-Wu and Kumar-Pragada used triangle-counting to prove special cases of the Bollob\'{a}s-Nikiforov conjecture.

We prove the next theorem with the triangle-counting method. We should mention that R\"{a}ty, Sudakov and Tomon has essentially completed the arguments for this theorem in their beautiful work on discrepancy \cite{RST23}. Their proof involves the positive and negative square energies, which they denote by
$$s^+(G) := \sum_{i = 1}^K \lambda_i^2 \text{ and }s^-(G) := \sum_{i = K+1}^n \lambda_i^2$$
though they didn't refer to these quantities by name, and I couldn't find the theorem explicitly stated in their paper. For the sake of completeness, we extract the relevant arguments in Section 5 of \cite{RST23} to form a short and self-contained proof of this result.
\begin{theorem}
\label{thm:main-3}
For any graph $G$ with $n$ vertices and $m$ edges, we have $$s^+(G) \geq \frac{m^{4/3}}{n^{1/3} \lambda_1^{2/3}}.$$    
\end{theorem}
We mention some corollaries. First, we resolve a question of Elphick and Linz on the ratio between $s^+(G)$ and $s^-(G)$. If $G$ is an $n$-vertex graph, then the ratio $\frac{s^+(G)}{s^-(G)}$ can be as large as $(n - 1)$, as evidenced by the complete graph $K_n$. Elphick and Linz asked whether the reciprocal $\frac{s^-(G)}{s^+(G)}$ can be equally as large. They show \cref{thm:quadrangle-construction} as an example with $\frac{s^-(G)}{s^+(G)} = \Omega(n^{1/4})$. We show that this is tight up to a constant.
\begin{corollary}[\cite{EL24}, Question 1]
\label{cor:main-3-1}
    For any graph $G$ on $n$ vertices, we have $$\frac{s^-(G)}{s^+(G)} \leq 2n^{1/4}.$$
\end{corollary}
Second, we consider regular graphs. The original proof of \cref{conjecture:EFGW} for regular graphs, given in \cite[Theorem 8]{EFGW16}, shows the bound $s^+(G) \geq n$ for any degree $k \geq 3$. We now improve this bound.
\begin{corollary}
\label{cor:main-3-1-regular}  
If $G$ is $k$-regular on $n$ vertices, then 
$$s^+(G) \geq (k/4)^{2/3} n.$$
\end{corollary}
In contrast, no such bound is possible for $s^-(G)$, as evidenced by a disjoint union of cliques. In the regime $k \leq n^{3/4}$, this bound is tight up to a constant: for each prime $q$, we take $k = q(q^2 + 1)$, and take a disjoint union of the graphs in \cref{thm:quadrangle-construction}. In the regime $k > n^{3/4}$, the bound $s^+(G) \geq \lambda_1^2 = k^2$ is tighter. 

We can trivially bound $s^+(G) \leq k^2 + \lambda_2^2 n$. Thus, when $k = o(n^{3/4})$, \cref{cor:main-3-1-regular} gives $\lambda_2 = \Omega(k^{1/3})$. This recovers the \emph{dense Alon–Boppana bound} obtained by Balla in \cite{balla2024} using sophicated matrix projection techniques. We refer to Balla's paper for a nice overview on the history of this bound. The idea of using triangle-counting to prove the dense Alon-Boppana bound has already been realized by R\"{a}ty-Sudakov-Tomon \cite{RST23}, who proved the stronger bound $\lambda_2 = \Omega(n / k)$ when $k \in [n^{2/3}, n^{3/4}]$ and $\lambda_2 = \Omega(k^{1/2})$ when $k \leq n^{2/3}$. 

The Alon-Boppana bound is only stated for regular graphs in \cite{RST23}. For the sake of completeness, we record an Alon–Boppana type bound for non-regular graphs as well.
\begin{corollary}
\label{cor:alon-boppana-irregular}
For any graph $G$ on $n$ vertices with $m$ edges, if $\lambda_1 \leq m^{1/2} (2n)^{-1/8}$, then we have 
$$\lambda_2 = \Omega\left(\frac{m^{2/3}}{n^{2/3} \lambda_1^{1/3}}\right).$$    
\end{corollary}
We can check that \cref{thm:main-3} establishes \cref{thm:main-result-2} for regular and almost-regular graphs. Unfortunately, it does not imply \cref{thm:main-result-2} for highly irregular graphs, where $\lambda_1$ can be much greater than the average degree bound $\frac{2m}{n}$. To complete the proof of \cref{thm:main-result-2}, we use an additional result that relates the square energies to the surplus of a graph.

The \emph{maxcut} of a graph $G$, denoted by $\mc(G)$, is the size of the largest bipartite subgraph in $G$. In the past 50 years, this quantity has been extensively studied both in Theoretical Computer Science and in Extremal Graph Theory. The \emph{surplus} of a graph $G$ is defined by $\surp(G) = \mc(G) - m / 2$. It is well known that $\mc(G) \geq m / 2$, so the surplus is always non-negative. We refer the reader to a beautiful work by Balla, Janzer and Sudakov \cite{BJS} for an overview of developments concerning $\surp(G)$.

R\"{a}ty, Sudakov and Tomon \cite[Lemma 5.7]{RST23} showed that positive discrepancy, which is closely related to the surplus, is lower-bounded by $\frac{s^+(G) - \lambda_1^2}{\sqrt{\Delta}}$. We show a complementary upper bound of the surplus by the square energies.
\begin{theorem}
    \label{thm:main-4}
    Let $G$ be a graph with $m$ edges. We have
    $$\min(s^+(G), s^-(G)) \geq \frac{\surp(G)^2}{m}.$$
\end{theorem}
I don't believe that this bound is tight for the full range of $\surp(G)$ and $m$, and conjecture some potential improvements to this result in \cref{sec:conclusion}. A combination of \cref{thm:main-3} and \cref{thm:main-4}, together with a vertex partition strategy, completes the proof of \cref{thm:main-result-2}.

\subsection*{Organization}
In \cref{sec:prelim}, we introduce a few notations and basic spectral properties that we will use throughout the paper. Subsequently, \cref{sec:additive}, \ref{sec:removal}, \ref{sec:triangle}, \ref{sec:surplus} contain the proofs of \cref{thm:main-1}, \ref{thm:main-2}, \ref{thm:main-3}, \ref{thm:main-4} respectively. Each section also contain the short proofs of the corollaries associated with each theorem. We complete the proof of \cref{thm:main-result-1} at the end of \cref{sec:removal}, and the proof of \cref{thm:main-result-2} at the end of \cref{sec:surplus}. Finally, \cref{sec:conclusion} provides a number of open problems for further study.

\subsection*{Acknowledgements}
I thank Gabriel Coutinho and Thom\'{a}s Spier for numerous helpful discussions, introducing me to the semi-definite programming approach and sharing notes on \cref{lem:energy-sdp-2}. I also thank Clive Elphick for his very nice lecture at University of Waterloo, where he provided a beautiful survey of the backgrounds and progresses on \cref{conjecture:EFGW}, and for providing valuable feedbacks on this manuscript like suggesting and proving the lower bound of $n^0$ in \cref{thm:main-result-1-weak}, as well as suggesting \cref{conj:elphick}.

\section{Notations and Preliminary Results}
\label{sec:prelim}
Throughout the paper, let $G$ be a simple graph with vertex set $V$ and edge set $E$. Let $n$ denote the number of vertices in $G$, and $m$ denote the number of edges in $G$. For a vertex subset $U \subset V$, let $G[U]$ denote the subgraph of $G$ induced by $U$, and let $G \backslash U$ denote the subgraph of $G$ induced by $V \backslash U$.

Let $A_G$ denote the adjacency matrix of $G$, whose rows and columns are indexed by $V$. Let $\lambda_1 \geq \cdots \geq \lambda_n$ denote the eigenvalues of $A_G$, ordered from largest to smallest. 

Let $\norm{\cdot}$ denote the Frobenius norm $\norm{M} := \sqrt{\tr(M^2)}$ on the space of real, symmetric matrices of a given dimension. Let $\langle \cdot, \cdot \rangle$ denote the corresponding inner product $\langle A, B \rangle := \tr(A B).$ For a square matrix $M$ and a set of indices $I$, let $M|_I$ denote the $I \times I$ submatrix $(M_{ij})_{i, j \in I}$ of $M$.

We define two matrices $A_G^+$ and $A_G^-$ as follows. We have a spectral decomposition
$$A_G = \sum_{\lambda} \lambda E_{\lambda}$$
where $\lambda$ goes over the eigenvalues of $A_G$, and $E_{\lambda}$ is the orthogonal projection onto the $\lambda$-eigenspace of $A_G$. We define
$$A_G^+ = \sum_{\lambda > 0} \lambda E_{\lambda} \text{ and }A_G^- = \sum_{\lambda < 0} -\lambda E_{\lambda}$$
Then $A_G^+$ and $A_G^-$ are positive semi-definite matrices with
$$A_G = A_G^+ - A_G^-.$$
We note that $s^+(G) = \norm{A_G^+}^2$, $s^-(G) = \norm{A_G^-}^2$, and $\langle A_G^+, A_G^- \rangle = 0$. 

We will repeatedly use a basic fact about the square energy.
\begin{lemma}[\cite{EFGW16}, Theorem 7]
    If $G$ has $m$ edges, then $s^+(G) + s^-(G) = 2m$. If $G$ is in addition bipartite, then  $s^+(G) = s^-(G) = m$.
\end{lemma}
\begin{proof}
    The sum of squares of the eigenvalues of $A_G$ is equal to $\tr(A_G^2)$, which is equal to $2m$. So we have $s^+(G) + s^-(G) = 2m$.
    
    The spectrum of any bipartite graph is symmetric with respect to the origin, so $s^+(G) = s^-(G)$. Therefore, both $s^+(G)$ and $s^-(G)$ are equal to $m$.
\end{proof}
We also need some folklore results about the spectrum of specific graphs.
\begin{lemma}
The spectrum of the complete graph $K_n$ is $n - 1, -1, -1, \cdots, -1$, where there are $(n - 1)$ copies of $-1$. In particular, we have $s^+(K_n) = (n - 1)^2$ and $s^-(K_n) = n - 1$.

The spectrum of $K_{1, n - 1}$ is $\sqrt{n - 1}, -\sqrt{n - 1}$, plus $(n - 2)$ copies of $0$.
\end{lemma}
We call a graph of the form $K_{1, n - 1}$ a \emph{star}.

At one point, we need the spectrum of the graph $U_{n, 3}$ introduced by Abiad et. al. in \cite{Abiad23}.
\begin{lemma}
\label{lem:spectrum-Un3}
Let $U_{n, 3}$ be the $n$-vertex graph obtained by adding an edge between two leaves of the star $K_{1, n - 1}$. Then the least eigenvalue of $U_{n, 3}$ satisfies $\lambda_n(U_{n, 3}) \leq -\sqrt{n- 2}$, with equality if and only if $n = 3$.
\end{lemma}
\begin{proof}
In \cite[Proposition 3.3]{Abiad23}, Abiad et. al. gave a bound on $\lambda_{n}(U_{n, 3})$. Unfortunately, their computation has a minor flaw. In their notation, $d_2$ is equal to $\frac{2}{n - 1}$ instead of $\frac{1}{n - 1}$. Once this is corrected, we have
$$\lambda_{n}(U_{n, 3})^2 \geq \frac{1}{4}\left(\sqrt{4(n - 1) + \frac{4}{(n - 1)^2}} - \frac{2}{n - 1}\right)^2 = \left(\sqrt{(n - 1) + \frac{1}{(n - 1)^2}} - \frac{1}{n - 1}\right)^2$$
which we can directly check is always at least $(n- 2)$, with equality if and only if $n =3$.
\end{proof}
\section{Super-additivity of square energy}
\label{sec:additive}
In this section, we prove \cref{thm:main-1} and its corollaries. We are motivated by the semi-definite optimization approach to square energy discussed in \cite{Coutinho2023}. Crucially, we describe $s^+(G)$ and $s^-(G)$ as the objective of an optimization program. Let $\cS_V$ denote the cone of real, positive semi-definite matrices with rows and columns indexed by $V$.
\begin{lemma}
    \label{lem:energy-sdp}
    Let $G$ be any graph. Then we have 
    $$s^+(G) = \min_{M \in \cS_V} \norm{A_G + M}^2$$
    and
    $$s^-(G) = \min_{M \in \cS_V} \norm{A_G - M}^2.$$
\end{lemma}
\begin{proof}
    For convenience, write $A^+ = A^+_G$ and $A^- = A^-_G$.
    
    For any $M \in \cS_V$, we observe that
    $$\norm{A_G + M}^2 = \norm{A^+ + (M - A^-)}^2 = \norm{A^+}^2 + 2
    \langle A^+, M - A^- \rangle + \norm{M - A^-}^2.$$
    Observe that $\langle A^+, A^- \rangle = 0$ and $\langle A^+, M \rangle \geq 0$ since both $A^+$ and $M$ are positive semi-definite. Thus we have $\langle A^+, M - A^- \rangle \geq 0$. So we conclude that 
    $$\norm{A_G + M}^2 \geq \norm{A^+}^2 = s^+(G)$$
    with equality when $M = A^-$. Therefore, we have the first identity $s^+(G) = \min_{M \in \cS_V} \norm{A_G + M}^2$.

    The proof of the second identity follows analogously by switching the roles of $A^+$ and $A^-$. We observe that
    $$\norm{A_G - M}^2 = \norm{A^- + (M - A^+)}^2 = \norm{A^-}^2 + 2
    \langle A^-, M - A^+ \rangle + \norm{M - A^+}^2.$$
    We use the facts that $\langle A^+, A^- \rangle = 0$ and $\langle A^-, M \rangle \geq 0$ since both $A^-$ and $M$ are positive semi-definite. Thus we have $\langle A^-, M - A^+ \rangle \geq 0$. So we conclude that 
    $$\norm{A_G + M}^2 \geq \norm{A^-}^2 = s^-(G)$$
    with equality when $M = A^+$. Therefore, we have $s^-(G) = \min_{M \in \cS_V} \norm{A_G - M}^2$.
\end{proof}
\cref{thm:main-1} follows almost immediately from this characterization of square energy.
\begin{proof}[Proof of \cref{thm:main-1}]
    For any real, symmetric $V \times V$ matrix $X$, observe that
    $$\norm{X}^2 = \sum_{i, j \in V} X_{ij}^2 \geq \sum_{i, j \in U} X_{ij}^2 + \sum_{i, j \in W} X_{ij}^2 = \norm{X|_U}^2 + \norm{X|_W}^2$$
    Therefore, for any matrix $M \in \cS_V$, we note that 
    $$\norm{A_G + M}^2 \geq \norm{(A_G + M)|_U}^2 + \norm{(A_G + M)|_W}^2 = \norm{A_{G[U]} + M|_U}^2 + \norm{A_{G[W]} + M|_W}^2.$$
    Since $M|_U$ and $M|_V$ are also positive semi-definite matrices, by \cref{lem:energy-sdp} we have
    $$\norm{A_{G[U]} + M|_U}^2 + \norm{A_{G[W]} + M|_W}^2 \geq s^+(G[U]) + s^+(G[W]).$$
    Applying \cref{lem:energy-sdp} again, we conclude that
    $$s^+(G) = \min_{M \in \cS_V} \norm{A_G + M}^2 \geq s^+(G[U]) + s^+(G[W]).$$
    The proof for $s^-(G)$ follows analogously by replacing all the $+$'s in front of $M$ with $-$'s.
\end{proof}
We now demonstrate the various corollaries of this result. The proofs of \cref{cor:main-1-2} and \cref{cor:main-1-3} are straightforward.
\begin{proof}[Proof of \cref{cor:main-1-2}]
Suppose $G$ is unicyclic. Let $v$ be any vertex on the cycle $C$. Then $G \backslash \{v\}$ consists of exactly one connected component $V_1$ that contains $C \backslash \{v\}$. Let $V_2 = V \backslash V_1$. Then the induced subgraphs $G[V_1]$ and $G[V_2]$ are trees. Therefore, they satisfy $s^{\pm}(G[V_1]) = \abs{V_1} - 1$ and $s^{\pm}(G[V_2]) = \abs{V_2} - 1$ by symmetry of the spectrum. We conclude by \cref{thm:main-1} that
$$s^+(G) \geq s^{+}(G[V_1]) + s^{+}(G[V_2]) = \abs{V_1} + \abs{V_2} - 2 = n - 2$$
and
$$s^-(G) \geq s^{-}(G[V_1]) + s^{-}(G[V_2]) = \abs{V_1} + \abs{V_2} - 2 = n - 2.$$
\end{proof}
\begin{proof}[Proof of \cref{cor:main-1-3}]
    Since each $G[V_i]$ is a bipartite graph with at least $\abs{V_i}$ edges, we have $s^{\pm}(G[V_1]) = \abs{E[V_i]} \geq \abs{V_i}$.
    By \cref{cor:main-1-restate}, we conclude that
    $$s^+(G) \geq \sum_{i = 1}^k s^{+}(G[V_i]) = \sum_{i = 1}^k \abs{V_i} = n$$
    and
    $$s^-(G) \geq \sum_{i = 1}^k s^{-}(G[V_i]) = \sum_{i = 1}^k \abs{V_i} = n$$
    as desired.
\end{proof}

To show \cref{cor:main-1-1}, we need a decomposition of the graph $G$ into stars and complete graphs.
\begin{lemma}
\label{lem:decompose-stars-cliques}
    Let $G$ be a graph without isolated vertices. Then there exists a partition $V = V_1 \sqcup \cdots \sqcup V_k$ of the vertex set such that for each $i$, $V_i$ contains at least two vertices, and $G[V_i]$ is either a star or a clique.
\end{lemma}
\begin{proof}
We argue by induction on the number of vertices in $G$. When $G$ has $2$ vertices, the result is trivial since $G = P_2$ and $V$ itself is such a partition.

Suppose the result holds for all graphs with less vertices than $G$. Let $G_1$ be any connected component of $G$, and let $T$ be any spanning tree of $G_1$. We consider a vertex $u \in T$ that is adjacent to at least one leaf of $T$. Let $L$ be the set of leaves adjacent to $V$ in $T$. We divide into three cases.

\textbf{Case 1: }$G_1$ has at most $3$ vertices. We take $V_1$ to be the vertex set of $G_1$. Since $G_1$ is connected, it is either a star $K_{1, 2}$ or a clique $K_2$ or $K_3$. Furthermore, $G \backslash V_1$ has no isolated vertices, so we can apply the induction hypothesis on $G \backslash V_1$ to partition the remaining vertices.

\textbf{Case 2: }$G_1$ has at least $4$ vertices, and there is an edge in $G$ between two vertices $v_1$ and $v_2$ in $L$. We let $V_1 = \{v_1, v_2\}$. Then $G[V_1]$ is a clique on $2$ vertices. Furthermore, $G_1 \backslash V_1$ is a connected graph with at least two vertices, so it has no isolated vertices. Thus $G \backslash V_1$ has no isolated vertices, so we can apply the induction hypothesis on $G \backslash V_1$ to partition the remaining vertices.

\textbf{Case 3: }There is no edge in $G$ between any pair of vertices in $L$. Let $V_1 = \{u\} \cup L$. Then $G[V_1]$ is a star. Furthermore, $G_1 \backslash V_1$ is a connected graph. If it has exactly one vertex $a$, then $a$ must be a leaf in $T$ and adjacent to $u$, contradicting the definition of $L$. Therefore, $G_1 \backslash V_1$ is either empty or has at least two vertices. In either case, $G \backslash V_1$ has no isolated vertices, so we can apply the induction hypothesis on $G \backslash V_1$ to partition the remaining vertices. 
\end{proof}
This decomposition leads to the proof of \cref{cor:main-1-1}.
\begin{proof}[Proof of \cref{cor:main-1-1}]
    Let $V = V_1 \sqcup \cdots \sqcup V_k$ be a partition of the vertex set of $G$ given by \cref{lem:decompose-stars-cliques}. Since stars and cliques satisfy \cref{conjecture:EFGW}, for each $i$ we have
    $$\min(s^+(G[V_i]), s^-(G[V_i])) \geq \abs{V_i} - 1 \geq \frac{1}{2} \abs{V_i}.$$
    By \cref{cor:main-1-restate}, we conclude that
    $$\min(s^+(G), s^-(G)) \geq \sum_{i = 1}^k \min(s^+(G[V_i]), s^-(G[V_i])) \geq \sum_{i = 1}^k \frac{1}{2} \abs{V_i} = \frac{n}{2}$$
    as desired.
\end{proof}
\section{Removing one vertex}
\label{sec:removal}
This section is devoted to \cref{thm:main-2} and its corollaries, culminating in the proof of \cref{thm:main-result-1}. Unlike the relatively computation-free proof of \cref{thm:main-1}, the proof of \cref{thm:main-2} relies on a peculiar numerical coincidence.

For a square matrix $M$ and row index $i$, let $s_i(M)$ denote the sum of squares of all entries of $M$ that are either on the $i$-th row or on the $i$-th column. If $M|_{-i}$ denote the matrix obtained by removing the $i$-th row and $i$-th column from $M$, then we have $\norm{M}_F^2 \geq \norm{M|_{-i}}_F^2 + s_i(M).$ The key lemma is
\begin{lemma}
\label{lem:3by3}
Let $A = \begin{bmatrix}
    0 & 1 & 0 \\
    1 & 0 & 1 \\
    0 & 1 & 0
\end{bmatrix}$ be the adjacency matrix of $P_3$. Let $M$ be any positive semi-definite $3 \times 3$ matrix. Then there exists an $i \in \{1,2,3\}$ such that $s_i(A - M) > 1$, and there exists a $j \in \{1,2,3\}$ such that $s_j(A + M) > 1$.\footnote{The constant $1$ here is not optimal. It can be replaced with something slightly larger than $1$.}
\end{lemma}
\begin{proof}
    Suppose for the sake of contradiction that for each $i \in \{1,2,3\}$, we have $s_i(A - M) \leq 1$. Assume
    $$M = \begin{bmatrix}
    a & b & d \\
    b & c & e \\
    d & e & f
\end{bmatrix}.$$
Then we have
$$A - M = \begin{bmatrix}
    -a & 1-b & -d \\
    1-b & -c & 1-e \\
    -d & 1-e & -f
\end{bmatrix}.$$
so we must have
$$\begin{cases}
    a^2 + 2(1 - b)^2 + 2d^2 \leq 1 \\
    f^2 + 2(1 - e)^2 + 2d^2 \leq 1 \\
    c^2 + 2(1 - b)^2 + 2(1 - e)^2 \leq 1
\end{cases}.$$
Set $x = (b + e) / 2$. Using the AM-GM inequality and the third inequality, we have
$$c^2 \leq 1 - 2(1 - b)^2 - 2(1 - e)^2 \leq 1 - 4(1 - x)^2.$$
In particular, as $c^2 \geq 0$, we must have $x \geq 0.5$.

Adding the first two inequalities, we obtain
$$a^2 + f^2 + 4d^2 + 2(1 - b)^2 + 2(1 - e)^2  \leq 2.$$
Therefore we have
$$a^2 + f^2 + 4d^2 \leq 2 - 4(1 - x)^2 = 2(1 - 2(1 - x)^2).$$
Applying the Cauchy-Schwarz inequality, we obtain
$$(a + f + 2d)^2 \leq (a^2 + f^2 + 4d^2)(1 + 1 + 1) \leq 6(1 - 2(1 - x)^2).$$
We now use the assumption that $M$ is positive semi-definite. For any $\lambda \in \RR$, set $v_{\lambda} = \begin{bmatrix}
    1 \\
    -\lambda \\
    1
\end{bmatrix}$. We have $v_{\lambda}^T M v_{\lambda} \geq 0$, which implies
$$(a + f + 2d) - 2(b + e) \lambda + c \lambda^2 \geq 0.$$
Since this holds for any $\lambda \in \RR$, computing the discriminant of this quadratic polynomial gives
$$c(a + f + 2d) \geq (b + e)^2 = 4x^2.$$
Furthermore, looking at the top and bottom $2 \times 2$ submatrix of $M$, we have $ac \geq b^2$ and $cf \geq e^2$. Since we must have $a,c,f \in [0, 1]$, we have $b, e \leq 1$, so $x = (b + e) / 2 \leq 1$.

Combining everything above, we have $x \in [0.5, 1]$ and
$$16 x^4 \leq c^2(a + f + 2d)^2 \leq 6(1 - 4(1 - x)^2) (1 - 2(1 - x)^2).$$
However, a direct computation via Desmos or Wolfram-Alpha shows that for $x \in [0.5, 1]$, we have (with a leeway of $0.5$)
$$16 x^4 > 6(1 - 4(1 - x)^2) (1 - 2(1 - x)^2)$$
which gives a contradiction. So there must exist some $i \in \{1,2,3\}$ such that $s_i(A - M) > 1$.

To prove the statement for $A + M$. Let $M' = DMD$, where $D$ is the diagonal matrix $\mathsf{diag}(1, -1, 1)$. Since $A = -DAD$, we have $(A + M) = D(A - M')D$, so we have $s_i(A + M) = s_i(A - M')$ for every $i \in \{1,2,3\}$. Since $M'$ is positive semi-definite, there exists an $i$ such that $s_i(A - M') > 1$, so $s_i(A + M) > 1$.
\end{proof}
With this lemma, we are ready to prove \cref{thm:main-2}.
\begin{proof}[Proof of \cref{thm:main-2}]
    By \cref{lem:energy-sdp}, there exists a positive semi-definite matrix $M$ such that $s^-(G) = \norm{A_G - M}^2$. As $A_{G[U]}$ is the adjacency matrix of $P_3$, by \cref{lem:3by3}, there exists some $u \in U$ such that $s_u(A_{G[U]} - M|_U) > 1$. Thus, we have
    $$\norm{A_G - M}^2 = \norm{A_{G\backslash\{u\}} - M|_{-u}}^2 + s_u(A_G - M) \geq \norm{A_{G\backslash\{u\}} - M|_{-u}}^2 + s_u(A_{G[U]} - M|_U).$$
    By \cref{lem:energy-sdp} again, we have $\norm{A_{G\backslash\{u\}} - M|_{-u}}^2 \geq s^-(G \backslash\{u\})$. So we conclude that $s^-(G) > s^-(G \backslash\{u\}) + 1$, as desired.

    The proof for $s^+$ follows analogously by replacing all $-$'s in front of $M$ with $+$'s.
\end{proof}
We now describe the strategy of \cref{cor:main-2-1}. Let $G$ be a graph with a dominating vertex $v_*$, and let $G'$ denote $G \backslash \{v_*\}$. If $G'$ contains an induced copy of $P_3$, then we can use \cref{thm:main-2} to remove a vertex in this path and decrease $s^-$ by more than $1$; note that this maintains the dominating vertex. Repeating this process, we eventually arrive at a graph where $G'$ does not have a copy of $P_3$. In other words, $G'$ is a disjoint union of cliques! Fortunately, this case can be dealt with using a generalization of the argument in \cite[Proposition 3.3]{Abiad23}.
\begin{lemma}
    \label{lem:universal-endgame}
    Let $G$ be a graph on $n$ vertices with a dominating vertex $v_*$ such that $G \backslash \{v_*\}$ is a disjoint union of cliques. Then $s^-(G) \geq n - 1$, with equality if and only if $G \cong K_{1, n - 1}$ or $G \cong K_{n}$.
\end{lemma}
\begin{proof}
    If $G \cong K_{1, n - 1}$ or $G \cong K_{n}$ the result follows by direct computation, so we assume that $G$ is neither.
    
    Let $\cC_1, \cdots, \cC_k$ be the cliques in $G \backslash \{v_*\}$. We first note that every vector $w: V \to \RR$ that satisfies
    $$\begin{cases}
        \sum_{v \in \cC_j} w_v = 0, \forall j \in [k] \\
        w_{v_*} = 0
    \end{cases}$$
    is an eigenvector of $A_G$ with eigenvalue $-1$. Computing the dimension of the $v$'s that satisfy these relations, $A_G$ has eigenvalue $-1$ with multiplicity at least $\sum_{i = 1}^k (\abs{\cC_i} - 1)$.

    On the other hand, $G$ contains $K_{1, k}$ as an induced subgraph by taking one vertex from each clique. Therefore, the least eigenvalue of $G$ is at most the least eigenvalue of $K_{1, k}$, which is $-\sqrt{k} < -1$.

    We conclude that $A_G$ has eigenvalue $-1$ with multiplicity at least $\sum_{i = 1}^k (\abs{\cC_i} - 1)$, together with the least eigenvalue which is at most $-\sqrt{k}$. So we conclude that
    $$s^-(G) \geq \sum_{i = 1}^k (\abs{\cC_i} - 1) + (-\sqrt{k})^2 = \sum_{i = 1}^k \abs{\cC_i} = n - 1$$
    as desired.

    For equality to hold, the least eigenvalue of $A_G$ must be exactly $-\sqrt{k}$. However, since $G$ is neither $K_{1, n - 1}$ nor $K_n$, $G$ must contain $U_{k + 2, 3}$ as a subgraph with $k \geq 2$. By \cref{lem:spectrum-Un3}, the least eigenvalue of $U_{k + 2, 3}$ is strictly less than $-\sqrt{k}$, so we must have $\lambda_{n}(G)^2 > k$. Thus, the equality in $s^-(G) \geq n - 1$ does not hold.
\end{proof}
\begin{proof}[Proof of \cref{cor:main-2-1}]
    Suppose for the sake of contradiction that $G$ is a counterexample to \cref{cor:main-2-1} with the smallest number of vertex. Let $v_*$ be the universal vertex of $G$, and let $G' = G \backslash \{v_*\}$.  If $G'$ is a disjoint union of cliques, then \cref{lem:universal-endgame} shows that $G'$ must satisfy \cref{cor:main-2-1}, contradiction.
    
    If $G'$ is not a disjoint union of cliques, then $G'$ contains an induced copy of $P_3$. By \cref{thm:main-2}, there exists a vertex $u$ in $G'$ such that $s^-(G) > s^-(G \backslash \{u\}) + 1$. Since $G \backslash \{u\}$ has $(n - 1)$ vertices and a dominating vertex $v_*$, the minimality of $G$ implies that $s^-(G \backslash \{u\}) \geq n - 2$. Thus, $s^-(G) > n - 1$, so $G$ satisfies \cref{cor:main-2-1}, contradiction. Thus, we conclude that \cref{cor:main-2-1} has no counterexample.
\end{proof}
The proof of \cref{thm:main-result-1} is now an easy combination of \cref{cor:main-2-1} and the decomposition approach in \cref{thm:main-1}.
\begin{proof}[Proof of \cref{thm:main-result-1} and \cref{thm:main-result-1-weak}]
    Let $D$ be the smallest dominating set in $G$. We introduce an arbitrary total ordering on $D$. For each $i \in D$, let $V_i$ be the vertex set consisting of $i$ itself, and every vertex $v \in V \backslash D$ such that $i$ is the smallest neighbor of $v$ in $D$. Then $\{V_i\}_{i \in D}$ forms a partition of $V(G)$. By \cref{thm:main-1}, we have
    $$\min(s^+(G), s^-(G)) \geq \sum_{i \in D} \min(s^+(G[V_i]), s^-(G[V_i])).$$
    Each $G[V_i]$ has a universal vertex $i$. \cref{cor:main-2-1}, combined with \cite[Remark 1.3]{Abiad23}, shows that 
    $$\min(s^+(G[V_i]), s^-(G[V_i])) \geq \abs{V_i} - 1.$$ 
    Thus, we conclude that
    $$\min(s^+(G), s^-(G)) \geq n - \abs{D}$$
    which proves \cref{thm:main-result-1}.

    The largest independent set in $G$ is always a dominating set. Therefore, we have $\gamma(G) \leq \alpha(G)$, where $\alpha(G)$ is the independence number of $G$. Furthermore, 
    it is familiar that $\max(n^+, n^-) \leq n -\alpha(G)$, so we have $n - \gamma(G) \geq \max(n^+, n^-)$. In addition, Abiad et. al. \cite[Theorem 3.1]{AAFM} has shown that $n - \gamma(G) \geq n^0$ for any connected graph $G$ with at least two vertices. Since both sides are additive under disjoint union, this result holds for any graph $G$ with no isolated vertex. So we conclude that $\min(s^+(G), s^-(G)) \geq n - \gamma(G) \geq \max(n^+, n^-, n^0)$, which proves \cref{thm:main-result-1-weak}. 
\end{proof}
Finally, \cref{cor:main-2-3} follows directly from \cref{thm:main-2}.
\begin{proof}[Proof of \cref{cor:main-2-3}]
    Let $U$ be the three consecutive vertices of degree $2$ on the cycle in $G$. Then $G[U]$ is isomorphic to $P_3$. So by \cref{thm:main-2}, there exists some $u \in U$ such that $s^-(G) > s^-(G \backslash \{u\}) + 1$. Since every vertex of $U$ has degree $2$, $G \backslash \{u\}$ must be a tree. So $s^-(G \backslash \{u\}) = n - 2$, thus $s^-(G) > n - 2 + 1 = n - 1$ as desired. The proof for $s^+(G)$ is analogous.
\end{proof}
\section{Triangle Counting}
\label{sec:triangle}
In this section, we prove \cref{thm:main-3} and its various corollaries. As we mentioned in the introduction, our argument for \cref{thm:main-3} is essentially a subset of the methods in \cite[Section 5]{RST23}. 
\begin{proof}[Proof of \cref{thm:main-3}]
Let $G$ be a graph with $n$ vertices and $m$ edges. If $s^-(G) \leq m$, then $s^+(G) \geq m$. Since $\lambda_1 \geq 1$ and $\lambda_1 \geq \frac{2m}{n}$, we have
$$\frac{m^{4/3}}{n^{1/3} \lambda_1^{2/3}}\leq \frac{m^{4/3}}{n^{1/3} \lambda_1^{1/3}} \leq \frac{m^{4/3}}{(2m)^{1/3}} \leq m$$
so \cref{thm:main-3} holds if $s^-(G) \leq m$. Therefore, we may assume that $s^-(G) \geq m$.

Let  $\cE = \sum_{i = 1}^n \abs{\lambda_i}$ be the energy of the graph. Since the sum of the eigenvalues is equal to zero, we have
$$\sum_{i: \lambda_i > 0} \lambda_i = \sum_{i: \lambda_i < 0} (-\lambda_i) = \frac{\cE}{2}.$$
Thus by the Cauchy-Schwarz inequality, we have
$$s^+(G) \geq \sum_{i: \lambda_i > 0} \lambda_i^2 \geq \frac{1}{n} \left(\sum_{i: \lambda_i > 0} \lambda_i\right)^2 = \frac{\cE^2}{4n}.$$
We now use triangle-counting to produce another lower bound for $s^+(G)$. By assumption, we have
$$\sum_{i: \lambda_i < 0} (-\lambda_i)^2 \geq m.$$
By the Cauchy-Schwarz inequality, we get
$$\sum_{i: \lambda_i < 0} (-\lambda_i)^3 \geq \frac{\left(\sum_{i: \lambda_i < 0} (-\lambda_i)^2\right)^2}{\sum_{i: \lambda_i < 0} (-\lambda_i)} \geq \frac{2m^2}{\cE}.$$
Now recall that $\sum_{i = 1}^n \lambda_i^3$ is six times the number of triangles in $G$. In particular, it is non-negative. So we have
$$\sum_{i: \lambda_i > 0} \lambda_i^3 \geq \sum_{i: \lambda_i < 0} (-\lambda_i)^3 \geq \frac{2m^2}{\cE}.$$
All the positive eigenvalues of $G$ are at most $\lambda_1$.\footnote{We can improve this step by using the fact that all positive eigenvalues of $G$ aside from $\lambda_1$ are at most $\lambda_2$. Doing so can improve the bounds on $\lambda_2$, potentially recovering the results in \cite{RST23} when $k \ll n^{3/4}$.} Therefore, we have
$$s^+(G) = \sum_{i: \lambda_i > 0} \lambda_i^2 \geq \frac{1}{\lambda_1} \sum_{i: \lambda_i > 0} \lambda_i^3 \geq \frac{2m^2}{\cE \lambda_1}.$$
Thus we obtain
$$s^+(G) \geq \max\left(\frac{2m^2}{\cE \lambda_1}, \frac{\cE^2}{4n}\right).$$
\cref{thm:main-3} follows by taking the geometric mean of the two bounds
$$s^+(G) \geq \left(\frac{2m^2}{\cE \lambda_1}\right)^{2/3} \left(\frac{\cE^2}{4n}\right)^{1/3} \geq \frac{m^{4/3}}{n^{1/3} \lambda_1^{2/3}}.$$
\end{proof}
\begin{proof}[Proof of \cref{cor:main-3-1}]
    We take the geometric average of \cref{thm:main-3} with the bound $s^+(G) \geq \lambda_1^2$ to obtain
    $$s^+(G) \geq \left(\frac{m^{4/3}}{n^{1/3} \lambda_1^{2/3}}\right)^{3/4} (\lambda_1^2)^{1/4} \geq \frac{m}{n^{1/4}} \geq \frac{s^-(G)}{2n^{1/4}}$$
    where we use the relation $s^-(G) \leq s^-(G) + s^+(G) = 2m$ in the last inequality.
\end{proof}
\begin{proof}[Proof of \cref{cor:main-3-1-regular}]
For a $k$-regular graph $G$, we have $\lambda_1 = k$ and $m = kn / 2$. Substituting these parameters into \cref{thm:main-3} gives $s^+(G) \geq (k/4)^{2/3}n$, as desired.
\end{proof}
\begin{proof}[Proof of \cref{cor:alon-boppana-irregular}]
    We have $s^+(G) \leq \lambda_1^2 + n \lambda_2^2$. By \cref{thm:main-3}, we obtain
    $$\lambda_2^2 \geq \frac{1}{n} \left(\frac{m^{4/3}}{n^{1/3} \lambda_1^{2/3}} - \lambda_1^2\right) \geq \Omega\left(\frac{m^{4/3}}{n^{4/3} \lambda_1^{2/3}}\right)$$
    as desired.
\end{proof}
\section{Surplus}
\label{sec:surplus}
The proof of \cref{thm:main-4} returns to the theme of optimization programs. In contrast to \cref{lem:energy-sdp}, we will write $s^+(G)$ and $s^-(G)$ as the objective of a maximization program, then substitute in an appropriate matrix to obtain \cref{thm:main-4}.
\begin{lemma}
    \label{lem:energy-sdp-2}
    Let $G$ be any graph. Then we have 
    $$s^+(G) = \max_{M \in \cS_V \backslash \{0\}} \frac{\max(\langle A_G, M \rangle, 0)^2}{\langle M, M \rangle}.$$
    and
    $$s^-(G) = \max_{M \in \cS_V \backslash \{0\}} \frac{\max(-\langle A_G, M \rangle, 0)^2}{\langle M, M \rangle}.$$
\end{lemma}
\begin{proof}
    We only prove the result for $s^+(G)$; the proof for $s^-(G)$ is symmetric. Let $m$ be the number of edges in $G$. By \cref{lem:energy-sdp}, we have
    $$s^+(G) = 2m - s^-(G) = \max_{M \in \cS_V} 2m - \norm{A_G - M}^2 = \max_{M \in \cS_V} 2 \langle M, A_G \rangle - \langle M, M \rangle.$$
    Since $\cS_V$ is a cone, we have $M \in \cS_V$ implies $\alpha M \in \cS_V$ for every $\alpha > 0$. Therefore, we have
    $$s^+(G) = \max_{M \in \cS_V \backslash \{0\}}  \max_{\alpha > 0} 2 \alpha\langle M, A_G \rangle - \alpha^2\langle M, M \rangle = \max_{M \in \cS_V \backslash \{0\}} \frac{\max(\langle A_G, M \rangle, 0)^2}{\langle M, M \rangle}$$
    as desired.
\end{proof}
We also need a fact about the adjacency matrix of bipartite graphs.
\begin{lemma}
\label{lem:bipartite-adjacency-entries}
For a bipartite graph $H$ with parts $X$ and $Y$, and any $x \in X, y \in Y$, the entry of $A_H^+$ at indices $(x, y)$ is equal to $\frac{1}{2}$ if $\{x, y\} \in E(H)$, and $0$ otherwise. The entry of $A_H^-$ at indices $(x, y)$ is equal to $-\frac{1}{2}$ if $\{x, y\} \in E(H)$, and $0$ otherwise.
\end{lemma}
\begin{proof}
Let $\lambda_1, \cdots, \lambda_k$ be the positive eigenvalues of $H$, and let $v_1, \cdots, v_k$ be the corresponding unit eigenvectors. Then we have
$$A^+_H = \sum_{i = 1}^k \lambda_i v_iv_i^T.$$
Let $w_i$ be the vector formed from $v_i$ by negating the entries with indices in $Y$. Then $w_i$ is the unit eigenvector corresponding to $-\lambda_i$, so we have
$$A^-_H = \sum_{i = 1}^k \lambda_i w_iw_i^T.$$
Comparing the entry at $(x, y)$, we observe that
$$(A^-_H)_{x, y} = -(A^+_H)_{x, y}.$$
On the other hand, we have
$$A^+_H - A^-_H = A_H.$$
So we conclude that
$$(A^+_H)_{x, y} = -(A^-_H)_{x, y} = \frac{1}{2} (A_H)_{x, y}$$
as desired.
\end{proof}
We can now prove \cref{thm:main-4}, the result on surplus.
\begin{proof}[Proof for \cref{thm:main-4}]
Let $H$ be the largest bipartite subgraph in $G$, with $k \geq m/2$ edges. Assume that its two parts are $U$ and $W = V \backslash U$. Then the edge set of $G$ is partitioned between the edge sets of $G[U]$, $G[W]$, and $H$, so
$$A_G = A_H + A_{G[U] \sqcup G[W]}.$$
For $s^+(G)$, we take $M = A_H^+$. As $H$ is bipartite and has $k$ edges, we have $\langle A_H, A_H^+ \rangle = \langle A_H^+, A_H^+ \rangle = s^+(H) = k$. Thus we obtain
$$\langle A_G, A_H^+ \rangle = k + \langle A_{G[U] \sqcup G[W]}, A_H^+\rangle.$$
Let $B_H$ be the $V \times V$ matrix obtained by replacing the entries of $A^+_H$ in $U \times W$ and $W \times U$ with zero. By the Cauchy-Schwarz inequality, we have
$$\langle A_{G[U] \sqcup G[W]}, A_H^+\rangle = \langle A_{G[U] \sqcup G[W]}, B_H\rangle \geq - \sqrt{\langle A_{G[U] \sqcup G[W]}, A_{G[U] \sqcup G[W]} \rangle \langle B_H, B_H \rangle}.$$
We have $\langle A_{G[U] \sqcup G[W]}, A_{G[U] \sqcup G[W]} \rangle = 2(m - k)$. By \cref{lem:bipartite-adjacency-entries}, we also have
$$\langle B_H, B_H \rangle = \langle A_H^+, A_H^+ \rangle - 2\sum_{u \in U, w \in W} (A_H^+)_{u, w}^2 = k - \frac{k}{2} = \frac{k}{2}.$$
So we obtain
$$\langle A_{G[U] \sqcup G[W]}, A_H^+\rangle \geq -\sqrt{(m - k)k}$$
and
$$\langle A_G, A_H^+ \rangle \geq k - \sqrt{(m - k)k} = \frac{(2k - m) \sqrt{k}}{\sqrt{k} + \sqrt{m - k}} \geq \frac{2k - m}{2} = \surp(G).$$
By \cref{lem:energy-sdp-2}, we conclude that
$$s^+(G) \geq \frac{\langle A_G, A_H^+ \rangle^2}{\langle A_H^+, A_H^+ \rangle} = \frac{ \surp(G)^2}{k} \geq \frac{\surp(G)^2}{m}$$
as desired. The proof for $s^-(G)$ proceeds verbatim, except we take $M = A_H^-$ instead, and prove
$$\langle A_G, A_H^- \rangle = -k + \langle A_{G[U] \sqcup G[W]}, A_H^-\rangle \leq -k+\sqrt{(m - k)k}.$$
\end{proof}
Finally, we are ready to tackle \cref{thm:main-result-2}. The strategy is as follows: we partition the graph into $O(\log\log m)$ groups of vertices with similar degrees. If there are a lot of edges inside any of these groups, we can use \cref{thm:main-3} to finish. Otherwise, there are few edges inside each groups, so there must be many edges between two of the groups. In this case, we finish by applying \cref{thm:main-4} to the subgraph induced by these two groups.
\begin{proof}[Proof of \cref{thm:main-result-2}]
Let $G$ be a graph with $m$ edges. Without loss of generality, we may assume that $G$ has no isolated vertices. By \cref{cor:main-1-1}, we have $s^-(G) \geq n / 2 = \Omega(m^{1/2})$. We next prove that 
$$s^+(G) = \Omega\left(\frac{m^{6/7}}{(\log\log m)^{8/3}}\right)$$
We divide the vertices into subsets $V_0, V_1, \cdots, V_{k-1}$, where $V_0$ is the set of vertices with degree at least $0.5m^{3/7}$, and for each $i \geq 1$, $V_i$ is the set of vertices with degree in the interval $\left[2^{-2^i}m^{3/7}, 2^{-2^{i - 1}}m^{3/7}\right)$. We observe that $k = O(\log \log m)$. Since the sum of degrees is equal to $2m$, we have
$$\abs{V_0} \leq \frac{2m}{0.5 m^{3/7}} \leq 4m^{4/7},$$
$$\abs{V_i} \leq \frac{2m}{2^{-2^i}m^{3/7}} \leq 2^{2^i + 1} m^{4/7}, \forall i \geq 1.$$
Let $m_i$ denote the number of edges in $G[V_i]$, and let $m_{i, j}$ denote the number of edges with one endpoint in $V_i$ and the other in $V_j$. We divide into cases based on the sizes of $m_i$ and $m_{i, j}$.

\textbf{Case 1: }We have $m_0 \geq \frac{m}{2k^2}$. By \cref{cor:main-3-1}, we obtain
$$s^+(G[V_0]) \geq \frac{m_0}{2 \abs{V_0}^{1/4}} = \Omega\left(\frac{m}{k^2 \cdot (m^{4/7})^{1/4}}\right) = \Omega\left(\frac{m^{6/7}}{k^2}\right).$$
\textbf{Case 2: }We have $m_i \geq \frac{m}{2k^2}$ for some $i \geq 1$. Observe that $\lambda_1(G[V_i])$ is at most the maximum degree of $G[V_i]$, which is at most $2^{-2^{i - 1}}m^{3/7}$. By \cref{thm:main-3}, we obtain
$$s^+(G[V_i]) \geq \frac{m_i^{4/3}}{\abs{V_i}^{1/3} \lambda_1(G[V_i])^{2/3}} \geq \frac{m^{4/3}}{4k^{8/3} \left(2^{2^i + 1} m^{4/7}\right)^{1/3} \left(2^{-2^{i - 1}}m^{3/7}\right)^{2/3}} = \Omega\left(\frac{m^{6/7}}{k^{8/3}}\right).$$
\textbf{Case 3: }We have $m_i < \frac{m}{2k^2}$ for every $i \geq 0$. Observe that
$$\sum_{i = 0}^{k-1} m_i + \sum_{0 \leq i < j < k} m_{i, j} = m.$$
Therefore, there exists some $0 \leq i < j < k$ such that $m_{i, j} \geq \frac{2m}{k^2}$. So the surplus of the graph $G[V_i \sqcup V_j]$ is at least
$$\surp(G[V_i \sqcup V_j]) \geq \frac{1}{2}(m_{i, j} - m_i - m_j) \geq \frac{m}{2k^2}.$$
So by \cref{thm:main-4}, we obtain
$$s^+(G[V_i \sqcup V_j]) \geq \frac{\surp(G[V_i \sqcup V_j])^2}{m} \geq \frac{m}{4k^4}.$$
In all cases, there exists some induced subgraph $G'$ of $G$ such that
$$s^+(G') = \Omega\left(\frac{m^{6/7}}{k^{8/3}}\right).$$
By \cref{thm:main-1}, we conclude that
$$s^+(G) \geq s^+(G') = \Omega\left(\frac{m^{6/7}}{(\log\log m)^{8/3}}\right)$$
completing the proof of \cref{thm:main-result-2}.
\end{proof}

\section{Concluding Remarks}
\label{sec:conclusion}
While our tools bring us closer to \cref{conjecture:EFGW}, there are certain graphs for which the conjecture still seems out of reach. A concrete example which we can study next is as follows.
\begin{conjecture}
    Let $G$ be the unicyclic graph obtained by attaching three trees to the three vertices of $K_3$. Then $\min(s^+(G), s^-(G)) \geq n - 1$.
\end{conjecture}
\cref{thm:main-2} does not help in this case since the removal of any non-leaf vertex disconnects the graph. Perhaps we can try some different graph operations that don't affect connectedness, such as contracting an edge. 

We consider another family of graphs, which are far from unicyclic, for which \cref{conjecture:EFGW} seems hard. We take a cycle on $k$ vertices, and glue a disjoint copy of $K_3$ on each vertex. We call the resulting graph $C_k^3$. 

\begin{figure}[h]
    \centering
    \includegraphics[width=0.5\linewidth]{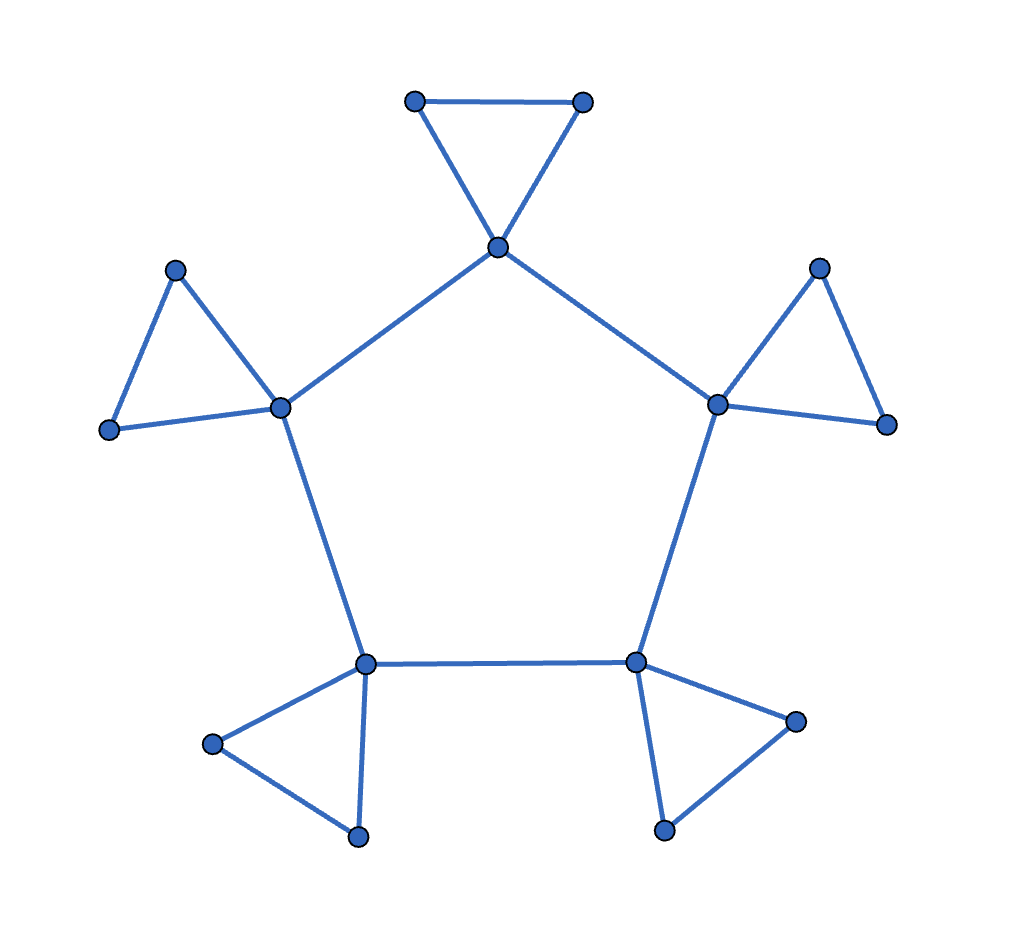}
    \caption{The graph $C_5^3$. Its negative square energy is $15.76$.}
    \label{fig:1}
\end{figure}

This graph has $n = 3k$ vertices and $m = 4k$ edges. By partitioning the graph into triangles, \cref{cor:main-1-restate} gives $s^+(C_k^3) \geq 4k$, but only gives $s^-(C_k^3) \geq 2k$.
Its clique number is $3$, so the bound of Ando-Lin gives us $\min(s^+(C_k^3), s^-(C_k^3)) \geq \frac{8k}{3} = \frac{8n}{9}$. Numerical evidence suggests that $s^-(C_k^3)$ is very close to $n$. For example, we have $s^-(C_{33}^3) \approx 104$. So \cref{conjecture:EFGW} holds just barely.

Nevertheless, \cref{thm:main-1} and \cref{thm:main-2} do put strong structural restrictions on any potential minimum counterexample. We observe that
\begin{proposition}
    Let $G$ be a counterexample to \cref{conjecture:EFGW} with the minimum number of vertices. Then $G$ must satisfy the following properties.

    (*): For any set of three vertices $U$ such that $G[U]$ is isomorphic to $P_3$, there exists a vertex $u \in U$ such that $G \backslash \{u\}$ is disconnected.

    (**): For any vertex subset $U$, if $G[U]$ is a bipartite graph with at least $\abs{U}$ edges, then $G \backslash U$ is disconnected.
\end{proposition}
Thus, it suffices to prove \cref{conjecture:EFGW} for all graphs that satisfy (*) and (**). It seems interesting to determine what these graphs can be.
\begin{problem}
    Classify all connected graphs with properties (*) and (**).
\end{problem}
As a strengthening of \cref{thm:main-result-1-weak}, Elphick communicated the following conjecture, which could serve as a next goal towards the solution of \cref{conjecture:EFGW} .
\begin{conjecture}[Elphick]
\label{conj:elphick}
Let $G$ be a graph with no isolated vertex. Let $(n^+, n^0, n^-)$ be the inertia of $A_G$. We have 
$$\min(s^+(G), s^-(G)) \geq n^0 + \max(n^+, n^-).$$
\end{conjecture}
Elphick points out that the inequality $n - \gamma(G) \geq n^0 + \max(n^+, n^-)$ does not always hold, so \cref{thm:main-result-1} does not imply this conjecture. 

Regarding \cref{thm:main-result-2}, it would be nice to get rid of the double logarithm term.
\begin{conjecture}
If $G$ is a graph with $m$ edges, then we have
$$s^+(G) = \Omega(m^{6/7}).$$
\end{conjecture}

Finally, we consider the relationship between squared energy and the surplus. The work of R\"{a}ty-Sudakov-Tomon \cite{RST23} and \cref{thm:main-4} indicate some relationships between $s^{\pm}(G)$ and discrepancies, so it is of interest to explore this connection more thoroughly. I conjecture that \cref{thm:main-4} can be improved as follows.
\begin{conjecture}
    For any graph $G$, we have
    $$s^+(G) = \Omega(\surp(G)).$$
\end{conjecture}
To provide a piece of evidence, this conjecture is true for the graph $G$ in \cref{thm:quadrangle-construction}. The expander mixing lemma shows that $\surp(G) \ll \max(\lambda_2, -\lambda_n) \cdot n = O(q^6)$, which matches with $s^+(G) \sim q^6$. Furthermore, Balla, Janzer and Sudakov \cite{BJS} showed that $\surp(G) \gg \frac{m}{\chi_v(G)}$, and Coutinho-Spier \cite{Coutinho2023} showed that $s^+(G) \geq \frac{2m}{\chi_v(G)}$, where $\chi_v(G)$ is the vector chromatic number. This common lower bound suggests that a linear relation could exist between the two quantities.

The analogous conjecture is not true for $s^-(G)$, even though Coutinho and Spier's bound does apply to $s^-(G)$. Let $H$ be any regular graph. Let $G$ be the join of two copies of the complement $\overline{H}$ of $H$. Let $J$ denote the all-one matrix. Then we have
$$J - A_G = \begin{bmatrix}
    A_H + I & 0 \\
    0 & A_H + I
\end{bmatrix}.$$
Therefore, if $\lambda$ is an eigenvalue of $H$, then $-\lambda-1$ appears twice in the spectrum of $G$ (or once if $\lambda$ is the first eigenvalue), and all eigenvalues of $G$ except the first eigenvalue arise this way. Thus, we can show that
$$s^-(G) = O(s^+(H)).$$
On the other hand, the surplus of $G$ is at least $\abs{H}$. In particular, taking $H$ to be the graph in \cref{thm:quadrangle-construction}, we have $s^-(G) = O(\surp(G)^{6/7})$. I conjecture that this is the worst case scenario.
\begin{conjecture}
For any graph $G$, we have
$$s^-(G) = \Omega(\surp(G)^{6/7}).$$    
\end{conjecture}
\bibliographystyle{plain}
\bibliography{bib.bib}

\end{document}